\documentclass[11pt, reqno]{amsart}
\usepackage{verbatim}
\usepackage{amssymb,amsmath,amsthm}
\usepackage{graphicx}
\usepackage{epsfig}

\newtheorem{theorem}{Theorem}[section]

\newtheorem{lemma}[theorem]{Lemma}
\newtheorem{proposition}[theorem]{Proposition}

\newtheorem{defn}[theorem]{Definition}
\newcommand{\Prob} {{\mathbb P}}
\newcommand{\Z}{{\mathbb Z}} 
\newcommand{\E}{{\mathbf E}}

\newcommand{\R}{{\mathbb{R}}}
\newcommand{\C}{{\mathbb C}}

\newcommand{\dist}{{\rm dist}}

\newcommand{\slit}{{U}^-}

\renewcommand{\Prob}{\mathbf{P}}

\def \Im {{\rm Im}}
\def \Re {{\rm Re}}
\def \p {\partial}

\def \Disk {{\mathbb D}}

\def \diam {{\rm diam}}

\def \eset {\emptyset}

\def \saws {{\cal R}}
\def \paths {{\cal K}}

\def \ee {\epsilon}

\def \linehere { {\hrule}}
\def \labove { \mtwo \linehere \linehere \linehere \ms   }

\def \lbelow {{\ms \linehere \linehere \linehere \mtwo}}

\def \mtwo {{\medskip \medskip}}
\def \ms {{\medskip}}
\newcommand {{\whoknows}} {{\mathcal C}}

\newenvironment{example}[1][Example.]{\begin{trivlist}
\item[\hskip \labelsep {\bfseries #1}]}{\end{trivlist}}

\newenvironment{advanced}
{ \labove \begin{quote} \begin{small}}
{ \end{small}\end{quote} \lbelow }

\def \begad{\begin{advanced}}
\def \endad{\end{advanced}}

\def \paths{{\mathcal K}}
\def \saws {{\mathcal W}}
\def \dlaplace {L}

\def \F {{\cal F}}
\newcommand{{\pe}}  {\partial_e}
\newcommand {{\lodd}} {{\mathcal J}}
\newcommand{{\inrad}} {{\rm inrad}}

\newcommand {{\cent}} {{w_0}}
\newcommand {{\eb}}  {{\bf e}}

\newcommand{{\saps}}  {{\mathcal X}}
\newcommand {{\dyadic}}  {{\mathcal Q}}
\newcommand {{\hm}} {{\rm hm}}
\newcommand {{\curves}} {\paths}
\newcommand {{\ball}} {{\mathcal B}}
\newcommand {{\measures}} {{\mathcal M}}
\newcommand{{\brown}}  {{\nu}}
\newcommand{{\osc}}  {{\rm osc}}
\newcommand  {{\rbrown}} {{\mu}}
\newcommand{{\lcurves}}{{\curves_L}}
\newcommand{{\bcurves} } {\curves^{\rm bub}}
\newcommand {{\bbrown}}{{\nu^{\rm bub}}}
\newcommand {{\n}} {{\bf n}}
\newcommand {{\timedense}} {\psi}
\newcommand {{\e}} {{\bf e}}
\title[Transition probabilities for two-sided LERW]{Transition probabilities for infinite two-sided loop-erased random walks}

\AtEndDocument{\bigskip{\footnotesize%

(Bene\v{s}) \textsc{Department of Mathematics, Brooklyn College, CUNY} \par  
  \textit{E-mail address}: \texttt{cbenes@brooklyn.cuny.edu} \par
    \addvspace{\medskipamount}

  (Lawler) \textsc{Department of Mathematics, University of Chicago} \par  
  \textit{E-mail address}: \texttt{lawler@math.uchicago.edu} \par
  
  \addvspace{\medskipamount}
  (Viklund) \textsc{Department of Matematics, KTH Royal Institute of Technology} \par
  \textit{E-mail address}: \texttt{fredrik.viklund@math.kth.se}
}}

\author{Christian Bene\v{s}, Gregory F. Lawler, Fredrik Viklund}
\begin{document}

\begin{abstract}
The infinite two-sided loop-erased random walk (LERW) is a measure on infinite self-avoiding walks that can be viewed as giving the law of the ``middle part'' of an infinite LERW loop going through $0$ and $\infty$. In this note we derive expressions for transition probabilities for this model in dimensions $d \ge 2$. For $d=2$ the formula can be further expressed in terms of a Laplacian with signed weights acting on certain discrete harmonic functions at the tips of the walk, and taking a determinant. The discrete harmonic functions are closely related to a discrete version of $z \mapsto \sqrt{z}$.

\end{abstract}


\maketitle
\section{Introduction}
By chronologically erasing loops from a simple random walk (SRW) one gets a measure on  self-avoiding walks (SAWs) called loop-erased random walk (LERW). There are several different versions depending on the boundary conditions of the SRW one starts with. In this note we will consider two-sided infinite volume versions on $\Z^d$ when $d \ge 2$. The infinite \emph{one-sided}  LERW from $0$ is defined by starting a SRW from $0$ and erasing loops as they form, and letting this process run. For $d \ge 3$ this works as stated but for $d=2$ because of recurrence, one needs to condition the random walk on never returning to $0$. (Technically, this conditioning, which is on a probability $0$ event, is done by weighting by the random walk potential kernel.) Here we will focus on a \emph{two-sided} version of the infinite LERW. We can think of this as giving the distribution of a ``middle part'' of an infinite LERW loop going through both $\infty$ and $0$. Defining this process requires work. Roughly speaking, one can consider chordal LERW conditioned to visit $0$ in successively larger domains,  and take an infinite volume limit. (Chordal LERW in a domain $A$ is obtained by loop-erasing a SRW starting and ending at given distinct boundary points, otherwise staying inside $A$.)
 A priori, any limit will depend on the sequence of domains, and showing that the limit is independent of this (and boundary conditions) is non-trivial.  We refer to \cite{Ltwo} for complete constructions of this process in dimensions $2$ and $3$, these being the more difficult cases, and \cite{L5d, LSW} for the remaining cases. We review parts of the story in the next section. Our main results concern transition probabilities for this model.

\begin{theorem}\label{thm:1}Let $d \ge 2$ and write $\hat{p} = \hat{p}_{0, \infty}^d$ for the infinite two-sided LERW probability measure, for SAWs on $\Z^d$ through $0$ and $\infty$ (see Theorem~\ref{Ltwo}). Let $\eta=[\eta_-, \ldots, \eta_+]$  be a finite SAW going through $0$. Then 
\[   \hat p(\eta) = (2d)^{-|\eta|} \cdot F_\eta \cdot \phi(\eta), \]
and  if  $\zeta \in \Z^{d}$ is a lattice neighbor of $\eta_{+}$, then in particular,
\[
\hat{p}(\eta^\zeta \mid \eta) = 
  \frac{\hat p(\eta^\zeta)}{\hat p(\eta)}  = \frac{1}{2d} \cdot G_{\Z^{d} \setminus \eta}(\zeta,\zeta)
     \cdot \frac{\phi(\eta^\zeta)}{\phi(\eta) }
 .
\]
Here $\eta^{\zeta}$ is the self-avoiding walk obtained by concatenating $\zeta$ to $\eta$, $F_\eta$ is the random walk loop measure of loops that stay in $\Z^d \setminus \{ 0 \}$ and intersect $\eta$, $G_{\Z^{d} \setminus \eta}$ is the random walk Green's function for $\Z^{d}\setminus \eta$, and $\phi$ is the limit of a particular escape probability for two independent random walks with a non-intersection condition given in Section~\ref{twosec}.
\end{theorem}
We will prove this theorem in Section~\ref{twosec}. In the plane we can give a different expression for the transition probability. The proof of this result requires some additional work which is done in Section~\ref{planesec}. In order to state the result, let us introduce a few additional functions; we refer to Section~\ref{defs} for more details.
First let $\tau_+$ be the first time SRW $S_j$ on $\Z^{2}$ visits $\mathbb{R}_+ = \{x: x \ge 0\}$ and for $R > 0$, let $\sigma_R = \min\{j
   \geq 0 : |S_j| \geq R\}$. Then we define
   \[ v(z) =   \lim_{R \rightarrow \infty} R^{1/2}
    \, \Prob^z\{ \sigma_R < \tau_+\} \]
    where $\Prob^z$ is the measure of SRW started at $z$. 
The existence of this limit (with an error estimate) is stated in Proposition~\ref{limitexists} which we prove in the Appendix. One can think of $v(z)$ as a discrete version of (a constant times) the imaginary part of $\sqrt{z}$. In fact, we show that $|v(z) -  \frac 4 \pi \, \Im \, \sqrt z| \le c/|z|^{1/2}$. Next, for a given SAW $\eta$ going through $0$
 let
\[   v_\eta(z) = v(z) - \sum_{y \in \eta} H^q_{\Z^2 \setminus \eta}\
   (z,y) \, v(y), \]
     where $H^q_{\Z^2\setminus\eta}$ is the Poisson kernel in $\Z^2\setminus\eta$ for a random walk with signed weights: The edge weights are $q(e) = \pm 1/4$, with negative sign iff $e=(k, k-i)$ with $k \in \mathbb{N}$. The functions $u, u_\eta$ are defined similarly replacing $v$ by a discrete harmonic conjugate. (So $u(z)$ is a discrete version of a constant times $\Re \, \sqrt{z}$.) See \eqref{v} and \eqref{veta} for the definitions.
\begin{theorem}\label{thm2}Consider the setting of Theorem~\ref{thm:1} and set $d=2$. If  $\eta_1=1$ and $\zeta \in \mathbb{Z}^2$ with
$|\eta_+-\zeta| = 1$,  then
\[  \hat p(\eta^\zeta \mid \eta) = \frac 14 \cdot G_{\Z^2 \setminus \eta}^q(\zeta,\zeta)
\cdot \left|\frac{D^q(\eta^\zeta)}{D^q(\eta)}\right|,\]
where
\[ D^q(\eta) = \det \left[\begin{array}{cc}
                L^qv_\eta(\eta_-) & L^qv_\eta(\eta_+)\\
                    L^q u_\eta(\eta_-) & L^q u_\eta(\eta_+)
                      \end{array} \right].\]
Here $G^{q}$ is the Green's function for random walk with signed weights $q$, $L^q$ is the corresponding  Laplacian, and $u_{\eta}, v_\eta$ are as above and defined precisely in Section~\ref{defs}.  
\end{theorem}

The definition of the signed weights depends on a choice of branch cut (or ``zipper'') just below the real axis. Other choices are possible, but the resulting formula in the theorem is independent of this choice.

Compare the theorem with the well-known representation of chordal two-dimensional LERW as Laplacian random walk, in the case when the walk runs from $a$ to $b$ (both boundary points) in a domain $A$: given $\eta$ starting from $a$, the probability that the next step is to $\zeta$   equals
\[
\frac{1}{4} \cdot \frac{H_{A \setminus \eta}(\zeta,b)}{LH_{A \setminus \eta}(\eta_+, b)} = \frac{1}{4} \cdot G_{A \setminus \eta}(\zeta, \zeta) \cdot  \frac{L H_{A \setminus \eta^\zeta}(\zeta,b)}{LH_{A \setminus \eta}(\eta_+, b)},
\] where $H_{A \setminus \eta}$ is the usual random walk Poisson kernel,  $L$ is the discrete (random walk) Laplacian here acting on the first variable, and $G$ is the random walk Green's function. In this case, the marked edge is on the boundary of the slit domain and there is no need to use signed weights. The reader may also compare the formula of Theorem~\ref{thm2} with Theorem 3.1 of \cite{BLV} which gives the probability that a chordal LERW visits a particular interior edge in terms of a determinant involving signed Poisson kernels. In fact, that formula is used in the proof of Theorem~\ref{thm2}.

In general it is not easy to compute explicitly $\hat p(\eta)$ for a given SAW $\eta$ using the techniques of this paper, but at the end of the paper we will provide one simple example showing that \begin{equation}\label{sr}
    \hat p(\eta_k) = \frac{1}{4} \, (\sqrt 2 - 1)^{k-1} \end{equation}
    in the case when $d=2$ and $\eta$ is the straight line $\eta = \eta_k  = [0,1,2\ldots,k]$.

This paper was motivated by a question from Kenyon and Wilson when they were preparing a paper on spanning trees with
a two-sided backbone \cite{KenWil} which also proved the formula \eqref{sr} using different techniques.  At the time, we thought that our results \cite{BLV} immediately gave a determinantal formula (a version of Fomin's identity) for the two-sided walk and \eqref{sr}.  It turned out that taking the limit, and, in particular, showing independence of boundary conditions required more work.  This paper was delayed until  Lawler adapted the proof in \cite{Ltwo}, which had originally been only for the three dimensional walk, to the two dimensional case to give the independence.  This allows us now to  build  on \cite{BLV}     to give the determinantal expression in Theorem 1.2.  Besides the fact that we employ different methods, a major difference between our paper and \cite{KenWil} is that we do not deal with the
full spanning tree.  It could however be constructed by starting with the two-sided loop-erased walk in \cite{Ltwo} and using Wilson's algorithm to complete the tree.  
%There is some overlap between this paper and \cite{KenWil}, but the set-up there is slightly different.
% we will not go into further details.

\section{Infinite two-sided LERW}  \label{twosec}
 We  now describe the infinite two-sided LERW centered at the origin as constructed in \cite{Ltwo}. A walk on $\mathbb{Z}^d$ is an ordered sequence of nearest neighbors. A self-avoiding walk (SAW) is a walk visiting each vertex at most once.
Let $\saws = \saws^d$ denote the set of SAWs on $\Z^d$ that
visit the origin.   We can write
\[  \saws = \bigcup_{j=0}^\infty \bigcup_{k=0}^\infty
\saws_{j,k}\] where $\saws_{j,k}$ is the set of
$\eta \in \saws$ with $j+k+1$ vertices that can be written as 
\[  \eta =  [\eta_{-j},\eta_{-j+1}, \ldots,
\eta_0 = 0 ,\eta_1,\ldots, \eta_k]. \]
Each element of $\saws_{j,k}$ can be viewed as a pair
of nonintersecting SAWs starting at the origin and traversing $j$ and $k$
edges, respectively.  If $\eta \in \saws$, we write $\eta_-,\eta_+$
for the initial and terminal vertices, i.e., if $\eta \in \saws_{j,k},$
then $\eta_- = \eta_{-j}$ and $\eta_+ = \eta_k$. 
If $\eta \in \saws_{j,k}$, we write $  \eta^o$ for the
corresponding SAW in $\saws_{0,j+k}$ obtained by
translating $\eta_{-}$ to the origin.  We write $\eta \prec \tilde \eta$ if
$\eta$ is a segment of $\tilde \eta$;  more precisely, if
$\eta \in \saws_{j,k}, \tilde \eta \in \saws_{l,
m}$ then $\eta \prec \tilde \eta$ if and only
if $j \leq l, k \leq   m, $ and
\[            \eta_n = \tilde \eta_n , \;\;\;\; n=-j,-j+1,
\dots, k.\]

The infinite two-sided LERW  
$ [\ldots, \hat S_{-2}, \hat S_{-1}, \hat S_0, \hat S_1,\hat S_2,
\ldots] $ 
is defined by giving the 
finite dimensional distributions 
\[   \hat p(\eta) := \Prob\{[\hat S_{-j},\ldots,\hat S_k]
 = \eta\}, \;\;\;\;\eta \in \saws_{j,k}. \]
They are defined via a limit established in \cite[Theorem~1]{Ltwo} which we discuss now. 
  Suppose $A $ is   simply
connected subset  of $\Z^d$ containing the origin.
  Let $a,b  \in \p A$ be such that there
there exists a self-avoiding path from $a$ to $b $, otherwise in $A $, that goes through the origin. 
Let $\eta \in  \saws$ with $\eta \subset A$. 
Let $S^1,S^2$ be independent simple random walks starting
at the origin and for $j=1,2$, $T_A^j  :=  \min\{t:
S_t^j \not \in A \}$.   Let
$   V_{A }, V_A(a,b)$ denote the non-intersection events
\[   V_A = \left\{LE(S^1[0,T_A^1]) \cap S^2[1,T_A^2]
 = \eset \right\}, \]
 \[    V_{A}(a ,b )  = V_A \cap \{S^1(T^1_A) = a ,
    S^2 (T^2_A) = b \}. \]
    Here $LE$ denotes chronological loop erasure. 
On the event $V_A $, we write  $\tilde \eta$   for the 
element of $\saws$  obtained by concatenating the reversal
of $LE(S^1[0,T_A^1])$ with $LE(S^2[0,T_A^2])$.
If $\eta \in \saws$, we let 
\[ V_A^\eta = V_A \cap \{\eta \prec \tilde \eta\},
\;\;\;\;   V_A^\eta(a,b) = V_A(a,b)
 \cap \{\eta \prec \tilde \eta\},\]   
\[    \hat p_A(\eta) = \frac{\Prob[V_A^\eta]}{\Prob[V_A]},
\;\;\;\;  \hat p_A(\eta; a,b) =
   \frac{\Prob[V_A^\eta(a,b)]}{\Prob[V_A(a,b)]}.\]
In other words, $\hat p_A(\eta; a,b)$ is the probability
 that a loop-erased walk from $a$ to $b$ in $A$ contains
 $\eta$ given that it goes through the origin. For $R > 0$, let $D_R = \{z\in\mathbb{Z}^2:|z|<R\}$.
    
\begin{theorem}[\cite{Ltwo}]  \label{Ltwo}
There exist $c < \infty, u > 0$ and
a function $\hat p: \saws
\rightarrow [0,1]$ such that
  the following holds.  
Suppose   $A$ is a simply connected
 subset of $\Z^d$  containing the origin  and $a,b \in \p A$
   with $\Prob[V_A(a,b)] > 0$.  Then for 
   all $\eta \in \saws$ with 
$\eta \subset D_r \subset D_n \subset A$ and $n \geq 2r$, 
 \[  \left|\log \hat p_A(\eta;a,b) - \log \hat p (\eta )
  \right| \leq c \,\left( \frac rn \right) ^u . \] 
\end{theorem} 

It follows immediately that
  \[  \left|\log \hat p_A(\eta ) - \log \hat p (\eta )
  \right| \leq c \,\left( \frac rn \right) ^u . \] 
We often write this as
\[     \hat p_A(\eta) = \hat p(\eta) \, \left[1 + O\left(\left( \frac rn \right) ^u\right) \right].\]
 By the definition and the theorem, we see that    the law  is reversible, that is, $\hat p(\eta) = 
 \hat p(\eta^R)$, where $\eta^R$ is the reversal of $\eta$,
 and translation invariant in the sense that $\hat p (\eta)
  = \hat p(\eta^o)$.
   Moreover, the law will have  the same reflection and rotation symmetries as $\Z^d$.  In particular, the trivial SAW
 of length $0$, $\eta = [0]$ satisfies $\hat p(\eta) = 1$ and
 $\hat{p}(\eta) = 1/2d$ if $|\eta| = 1$. 
 
Given the theorem, we can also express the probabilities
in terms of transition probabilities.  Since the
measure is invariant under translation, rotation, and
path reversal, it is sufficient to consider $\eta \in \saws_{0,k}$. 
If $|\zeta - \eta_+| = 1$, we write
$    \eta^\zeta =   \eta \oplus [\eta_+,\zeta]$ and 
\[ \hat p(\eta^\zeta \mid \eta)
    = \frac{\hat p(\eta^\zeta)}{\hat p(\eta)}
     = \lim_{n \rightarrow \infty}
      \frac{\hat p_{A_n}(\eta^\zeta)}{\hat p_{A_n}(\eta)}
       = \lim_{n \rightarrow \infty}
      \frac{\hat p_{A_n}(\eta^\zeta;a_n,b_n)}{\hat p_{A_n}(\eta ;a_n,b_n)} . \]
Here we write $A_n$ for a sequence of $A$ as in Theorem \ref{Ltwo} with $A_n \uparrow \mathbb{Z}^d$ as $n \to \infty$ and we write $a_n,b_n$ for boundary points of $A_n$. Theorem \ref{Ltwo} tells us that
the limits can be taken over any sequence of simply
connected $A_n$ with $\dist(0, \p A_n) \rightarrow \infty$
and $a_n,b_n \in \p A_n$ with $\Prob[V_{A_n}(a_n,b_n)] > 0$.
We will give the transition probability in terms of a
function $\phi$ that we will now define.  We will
define it for $\eta \in \saws_{0,j}$; it is extended
to other $\eta$ by the equality $\phi(\eta) = \phi(\eta^o)$.

  \begin{defn}
  Let $\eta 
   \in \saws_{0,j}$  with  $\eta^\zeta \subset A$. 
 \begin{itemize}
 \item {Let $S^1,S^2$ be independent simple random walks starting at $ \eta_+,0$, respectively; let $T_A^1,T_A^2$ be the first
 visits to $\p A$; and let $T_\eta^1, T_\eta^2$ the first visits
 to $\eta$ after time zero by $S^1,S^2$, respectively.}
 \item{Set $\omega^1=S^1[0,T_A^1]$ and
  $\omega^2=S^2[0,T_A^2]$
  % and  let $LE(\omega^1)$
 %denote the chronological loop erasure of $\omega^1$
 . }
 \end{itemize} Then we define: 
 \begin{equation}\label{Edef}    \phi_A(\eta )  = \Prob\{T_A^1 < T_\eta^1,
 T_A^2 < T_\eta^2, LE(\omega^1 ) \cap \omega^2 = \eset  \}, \end{equation}
  \[ \phi_A(\eta;z,w )  = \Prob\{T_A^1 < T_\eta^1,
 T_A^2 < T_\eta^2, LE(\omega^1 ) \cap \omega^2 = \eset,S^1_{T_A^1} = z, S^2_{T_A^2} = w\}.\]  
 \end{defn}
By splitting  
\[   \omega^1 = [\eta_+,\zeta]
 \oplus [\zeta,\ldots,\omega^1_k]
 \oplus [\omega^1_k, \ldots ,\omega^1_n] , \]
where $k$ is the largest index $i$ such that $\omega^1_i
 = \zeta$, we can see that
\begin{equation}\label{sept13}
   \hat p_A(\eta^\zeta\mid \eta) := 
 \frac{ \hat p_A(\eta^\zeta)}{ \hat p_A(\eta)} 
    = \frac{G_{A \setminus \eta}
 (\zeta,\zeta)}{2d}\, \frac{\phi_A(\eta^\zeta)}{
 \phi_A( \eta)} .  
 \end{equation}
 By iterating this, it follows that 
\begin{equation}  \label{may10.3}
     \hat p_A(\eta) = (2d)^{-|\eta|} \, F_\eta(\hat A) \,
           \frac{\phi_A(\eta)}{\phi_A(0) },
           \end{equation}
           and similarly
   \begin{equation}\label{may10.4}
               \hat p_A(\eta;z,w) = (2d)^{-|\eta|} \, F_\eta(\hat A) \,
           \frac{\phi_A(\eta;z,w)}{\phi_A(0;z,w) }.
           \end{equation}
 Here, and throughout, we write $\hat A = A \setminus \{0\},$
  $0$ is used to denote the trivial SAW of length zero,
 and
  \[     F_\eta(\hat A) =  \prod_{j=1}^{|\eta|} G_{A_j}(\eta_j,\eta_j),\]
  where $A_j = A \setminus \{\eta_0,\ldots,\eta_{j-1}\}$.
  Note that we can also write 
\[  F_\eta(\hat A)    = \exp \left\{\sum_l m(l)\right\}, \]
where $m$ denotes the random walk loop measure and the
sum is over all loops $l \subset \hat A$ that intersect
$\eta$. See Proposition 9.3.1 and Lemma 9.3.2 of \cite{LL}.
 
 Theorem \ref{Ltwo} says that we can
 take limits in \eqref{may10.3} and \eqref{may10.4}.
 Let $A_n$ be a sequence of simply connected sets containing
 the origin with $\dist(0, \p A_n) \rightarrow \infty$
 and let $a_n,b_n \in \p A_n$ with $\Prob[V_{A_n}(a_n,b_n)] > 0$.
We can then define
 \[    \phi(\eta) = \lim_{n \rightarrow \infty}
     \frac{\phi_{A_n}(\eta)}{\phi_{A_n}(0) }
  = \lim_{n \rightarrow \infty}
     \frac{\phi_{A_n}(\eta;a_n,b_n)}{\phi_{A_n}(0;a_n,b_n) }.\]
This definition, together with \eqref{may10.3} and \eqref{sept13}, implies Theorem~\ref{thm:1}.

 \section{Infinite two-sided LERW: The planar case}\label{planesec}
 
 In the case $d=2$, one can give a different expression for the transition probability. This will be a fairly simple consequence of work in \cite{BLV} combined with Theorem 1 in \cite{Ltwo} (from which Theorem \ref{Ltwo} was derived). An important ingredient is an idea of Kenyon to introduce a ``zipper'' or branch cut onto the lattice; in \cite{BLV} this was interpreted in terms of random walks and loop measures with weights negative on those edges crossing the zipper. Similar ideas are also important when analyzing the Ising model at criticality, see, e.g., \cite{CHI15}.   
 
  We will use complex notation and write $\Z^2 = \Z + i \Z$.  From here on, let 
\begin{equation}\label{A_n}
 A_n = \{x+iy \in \Z \times \Z: |x| < n\} , \;\;\;\;a_n = -n, b_n = n ,
 \end{equation}
denote the infinite strip and define $p^{(n)}(\eta)$ to be the probability that a LERW from $a_n$ to $b_n$ in $A_n$ includes $\eta$ conditioned on the event that the LERW passes through the
origin. It follows from Theorem \ref{Ltwo} that one has independence of boundary conditions so that
\[   \hat p(\eta) = \lim_{n \rightarrow \infty} p^{(n)}(\eta). \]
We will compute the transition probability by writing the expression for the strip $A_{n}$ and then taking the limit as $n \to \infty$.  This argument
will only use the expressions derived in \cite{BLV} and random walk estimates.

 We will first restate the main result for $d=2$, and then we will define the
 quantities in the statement.

\begin{proposition}  \label{main}  If $\eta \in \saws$ 
with   $\eta_1=1$,
$|\eta_+-\zeta| = 1$,
then
\[  \hat{p}(\eta^\zeta \mid \eta) = \frac 14 \cdot  G_{\Z^2\setminus\eta}^q(\zeta,\zeta)
 \cdot  \left|\frac{D^q(\eta^\zeta)}{D^q(\eta)}\right|,\]
where
\[ D^q(\eta) = \det \left[\begin{array}{cc}
                 L^qv_\eta(\eta_-) & L^qv_\eta(\eta_+)\\
                    L^q u_\eta(\eta_-) & L^q u_\eta(\eta_+)
                      \end{array} \right].\]
  \end{proposition}
  
We need to define $G_{\Z^2\setminus\eta}^q(\zeta,\zeta), L^q$
  and the functions in the determinant.  This will be done in the next subsection. The proof of  Proposition \ref{main} requires a number of estimates (provided in Section \ref{est-sec}) and will be completed in this section's third and last subsection. One of the intermediate results,  Proposition \ref{limitexists}, is of independent interest and will be proven in the Appendix for the sake of fluidity of this section%Section \ref{planesec}
.

 \subsection{Signed weights and discrete square root}\label{defs}

 We now introduce the key objects that appear in Proposition \ref{main}, namely $G^{q}, u, v,$ and $\dlaplace^q$. There is some arbitrariness in the definition of $u,v,q$ although the quantity $D^q(\eta)$
is independent of the choices. Our choices are natural given the particular embedding of the lattice into the complex plane.
Letting our domains $(A_n,a_n,b_n)$ be as in \eqref{A_n} is a  convenient
  choice that will allow us to make
  use of some symmetries.

 We will need to do some basic analysis
of the discrete Laplacian ``with zipper'' in $\Z^2$.   If $z,w \in \Z^2$ 
 let $p(z,w)$ be the usual simple random walk weight, i.e.,
\[p(z,w) = \frac{1}{4} \, {\bf 1}\{|z-w| = 1\}\] and $q(z,w)$ the weights obtained by
putting a branch cut or ``zipper'' directly below the positive real axis:
\[               q(k,k-i) = q(k-i,k) = - p(k,k-i) = -1/4, \;\;\;
  k=1,2,3,\ldots , \]
  and otherwise $q(z,w) = p(z,w)$.  
  One can think of $q$ as being the opposite of $p$ for all edges crossing the zipper $\{z\in\C:z=\frac{1}{2}-\frac{1}{2}i+x, x\in\R_+\}$ and otherwise the same as $p$.
  Given this, let us set some more notation.
  \begin{itemize}
  \item{In this section a walk is a sequence of nearest neighbor vertices on $\Z^2$ and a self-avoiding walk (SAW) is  a walk with no self-intersections. A loop is a walk starting and ending at the same vertex (with possible additional self-interesections.) The $p-$ and $q-$weights of a walk are the products of the weights of the traversed edges. The weight of the length-$0$ walk is by convention $1$.}
  \item{
  We write $\dlaplace$ and $\dlaplace^q$
  for the $p$- and $q$-Laplacians acting on functions defined on $\Z^{2}$:
  \[    \dlaplace f(z) =    \left[ \sum_{w: \, |w-z| =1} p(z,w) \, f(w) \right]
   - f(z) , \]
   \[  \dlaplace^q f(z) =  
       \left[ \sum_{w:\, |w-z| =1} q(z,w) \, f(w) \right]
   - f(z).  \]
   Note that the Laplacian of $f$ is also defined on $\Z^{2}$. We write $ \dlaplace_{x}f,  \dlaplace_{x}^{q}f$ for the Laplacians acting on the function $x \mapsto f(x,y)$.}
   \item{The $p$- and $q$-Green's functions for random walk in $A \subsetneq \Z^2$ are defined by
   \[
   G_{A}(z,w) = \sum_{\omega : z \leadsto w, \omega \subset A}p(\omega), \qquad G_{A}^{q}(z,w) = \sum_{\omega : z \leadsto w, \omega \subset A}q(\omega),
   \]
     where the sums are over all walks $\omega$
  starting at $z$, ending at $w$, otherwise in $A$.

   }
   \item{One can check that if $\eta$ is a SAW going through $0$, then
 \[  G_{\Z^2\setminus\eta}^q(w,w) = \sum_{\omega: w \leadsto
 w, \omega \subset \Z^2\setminus\eta} 4^{-|\omega|} \, (-1)^{J(\omega)},\]
 where the sum is over all random walk loops rooted at $w$ in $\Z^2\setminus\eta$ and
  $J(\omega)$ is the winding number of the loop about
 $0$ which is the same as the winding number about $\eta$, since the loops do not intersect it.
}

\item{
   If $A  \subsetneq \Z^2$, we define  $H_A(z,w), H_A^q(z,w)$ for $z \in \Z^2,
   w \in  \Z^2 \setminus A$ as the Poisson kernels corresponding to the respective Green's functions.   That is, if $z\in A$,
   
   \begin{equation}  \label{jan6.1}
    H_A(z,w) =  \! \sum_{\omega: z \leadsto w, \, \omega \setminus \{w\}  \subset  A } p(\omega), \quad
       H^q_A(z,w) = \!\sum_{\omega: z \leadsto w, \, \omega \setminus \{w\}  \subset  A } q(\omega) , 
       \end{equation}
  where the sums are over all walks $\omega$
  starting at $z$, ending at $w$,
  otherwise staying in $A$. 
  For $z \in \Z^2 \setminus A$, by definition $H_{A}(z,w) = H_{A}^q(z,w) =  \mathbf{1}_w(z)$ (a Dirac mass at $w$). More generally, if $E \subset \partial A$, then we  write
 $$H_{A}(z,E)=\sum_{w\in E}H_{A}(z,w)$$
 for the discrete harmonic measure of $E$ in $A$.
  
  }
  \item{   Note that for $w\in\Z^2\setminus A, h(z)   := H^q_{A}(z,w) $
   is the unique bounded function on $\Z^2$ with
  \[                  h(z) =  \mathbf{1}_w(z) , \;\;\;\; z \in \Z^2 \setminus A, \]
  \[       \dlaplace^q h(z) = 0 , \;\;\;\; z \in  A . \]
  Another way of looking at $h$ is as follows.  Let $S_n$ denote a simple random walk and let
  $Q_n$ denote the number of times that the walk has crossed the branch cut by time
  $n$.  If $\tau$ is the hitting time of $\Z^2 \setminus A$, then
  \[        h(z) = \E^z\left[ (-1)^{Q_\tau} \, \mathbf{1}\{S_\tau = w\} \right].\]
We can also think about this in terms of a two-cover of $\Z^{2}$.}
  \item{
       We also define Poisson kernels for two boundary points: If $z,w $ are distinct points in $\p A$, we
  define the boundary Poisson kernels,
  \[   H_{\p A}(z,w) =\dlaplace_z H_A(z,w) = \dlaplace_w H_{A}(z,w), \]
  \[
     H_{\p A}^q(z,w) = \dlaplace^q _zH_A(z,w)= \dlaplace^q_{w} H_A(z,w).\]
 We  
 can also write   $  H_{\p A}(z,w),  H_{\p A}^q(z,w)$
 in a form analogous to that in \eqref{jan6.1}.
}

\item{

We will in places need to compare discrete Poisson kernels with their continuous counterparts. For a domain $A\subset \C$, if $w\in A$, $z\in\partial A$, and $\p A$ is locally analytic at $z$, the Poisson kernel $h_A(w,z)$ is the density of harmonic measure with respect to Lebesgue measure.
As in the discrete case, we will write 
$$h_A(w, E)= \int_{E} h_A(w,z) \, |dz|$$
for domains $A$ with a Poisson kernel and measurable $E \subset \partial A$.
}

\item{ 
 
  We define \[\Z_+ = \{0,1,2,\ldots\}, \qquad \Z_+^{*} = \{1,2,\ldots\}\] and 
  \[\Z_-=\{\ldots, -2,-1,0\}, \qquad  \Z_-^{*} = \{\ldots, -2,-1\}.\] 
  
  }
  
\item{ For $R\in\R_+$, we let
  \begin{equation}\label{discretecircle}
D_R = \{z\in\mathbb{Z}^2:|z|<R\}, \quad C_R=\partial D_R
\end{equation}
be the discrete disk and circle, respectively, of radius $R$. Also let
  \begin{equation}\label{U}
U_R  = \{x+iy \in \Z^{2}: |x| < R, \, |y| < R \}
\end{equation}
be the discrete square centered at $0$ of side-length $2R$ with sides parallel to the axes.
}
 \end{itemize}
 
 \subsection{Random walk estimates}\label{est-sec}
 
 This section contains a number of useful estimates for random walks with or without signed weights. 

The first three estimates are for the Poisson kernel for walks with signed weights and exploit cancellations implied by the signed weights.
 \begin{lemma}  \label{prop.kias}
There exists $c < \infty$ such that the following is true.
Let 
$K \subset \Z^2$
be a finite set including the origin.  Then for all
$z$,
\[          \sum_{w \in  K}  |H_{\Z^2\setminus K }^q(z,w)| 
       \leq c \, \sqrt{\frac{\diam(K)}{|z|}} \]
 \end{lemma}
 
 \begin{proof}
 It suffices to prove the result in the case where $K = D_n$ with $n \in \mathbb{N}$.  Indeed, if
$ |z| \geq n$ and $K \subset \{|z| \le n\}$, then 
\[     H^q_{\Z^2\setminus K }(z,w) = \sum_{z' \in D_n}
    H^q_{\Z^2\setminus D_n}(z,z') \, H^q_{\Z^2\setminus K}(z',w).\]
    and hence, since $\sum_{w \in K} |H^q_{\Z^2\setminus K}(z',w)|\leq \sum_{w \in K}H_{\Z^2\setminus K}(z',w)\leq 1$,
  \begin{eqnarray*} 
  \sum_{w \in K}
   |H_{\Z^2\setminus K}^q(z,w)|  & \leq & 
       \sum_{z' \in D_n} \sum_{w \in K}
          |  H_{\Z^2\setminus D_n}^q(z,z') | \, |H_{\Z^2\setminus K}^q(z',w)| \\
       & \leq & 
    \sum_{z' \in D_n}  |  H_{\Z^2\setminus D_n}^q(z,z') | .
    \end{eqnarray*}
    Let $l, l'$ be the half-infinite lines
    \[   l = \{ ki: k \geq n \}, \;\;\;\;
       l' = - l =  \{- k i: k \geq n \}.\]
     Let $\tau = \min\{t: S_t \in D_n\}$,
     $T = \min\{t: S_t \in l\}, T' = \min\{t: S_t
      \in l'\}$.   On the event $\{T < T' < \tau\}$
     we can give a bijection on paths by switching
     $S[T,T']$ with its reflection about the
     imaginary axis.  The measure of the reflected
     path is the negative of the measure of the
     first path and hence these paths cancel.
     There is a similar bijection on
     $\{T' < T < \tau\}$.  Therefore, by the Beurling estimate,
     \begin{eqnarray*}     \sum_{z' \in D_n}
        |H^q_{\Z^2\setminus D_n}(z,z')| &  \leq &
          \Prob^z\{\tau < \max\{T,T'\}\} \\
         & \leq & \Prob^z\{\tau < T\}
            + \Prob^z\{ \tau < T'\}
              \leq  c \, \sqrt{\frac n{|z|}}.\end{eqnarray*}

    \end{proof}
    
      \begin{lemma}  \label{qBeurling}
There exists $c < \infty$ such that the following is true.
Let 
$K \subset \Z^2$
be a finite set including the origin, $z\in\Z^2\setminus K$ and let $r = \max\{|z|, \diam(K)\} $
.  Then 
\[          \sum_{x \in  C_{ar}}  |H_{D_{ar}\setminus K }^q(z,x)| 
       \leq c \, a^{-1/2}. \]
 \end{lemma}

\begin{proof}
Let $\psi = \inf\{k\geq 0: S(k)\in C_{ar}\}$ and $\lambda =\sup\{k\leq \psi: S(k)\in C_r\}$ and write $\mathbf{Q}^z$ for the $q$-measure of paths started at $z$. Then for $z\in\Z^2\setminus K, x \in  C_{ar}$,
\begin{eqnarray*}
|H^q_{D_{ar}\setminus K}(z,x)| & = & %|H^q_{D_{ar}\setminus D_r}(z,x)|+
|\sum_{w\in C_r}\sum_{k\geq 1}\mathbf{Q}^z(\lambda = k, S(\lambda)=w, S(\psi)=x)|\\
& = & %|H^q_{D_{ar}\setminus D_r}(z,x)|+
|\sum_{w\in C_r}\sum_{k\geq 1}\mathbf{Q}^z(\lambda = k, S(k)=w, S[k,\psi]\not\in D_r, S(\psi)=x)|\\
%& \leq & %|H^q_{D_{ar}\setminus D_r}(z,x)|+
%\sum_{w\in C_r}\sum_{k\geq 1}|\mathbf{Q}^z(\lambda = k, S(k)=w)| |\sup_{w\in C_r}\mathbf{Q}^w( S[0,\psi]\not\in D_r, S(\psi)=x)|\\
& \leq & %|H^q_{D_{ar}\setminus D_r}(z,x)|+ 
\sup_{w\in C_r}|H^q_{D_{ar}\setminus D_r}(w,x)|.\\
\end{eqnarray*}
So, writing for a set $A\subset \Z^2, \tau_A=\inf\{k\geq 0: S(k)\in A\}$ and using the same argument as in Lemma \ref{prop.kias} for the final inequality, 
$$\sum_{x \in  C_{ar}}  |H_{D_{ar}\setminus K }^q(z,x)| \leq %\sum_{x \in  C_{ar}}|H^q_{D_{ar}\setminus K}(z,x)|+ 
\sup_{w\in C_r}P^w(\tau_{C_{ar}}<\tau_{\Z_+^*})\leq ca^{-1/2}.$$
\end{proof}

  \begin{lemma}\label{withoutAn}  There exists $c < \infty$
  such that the following is true. Let 
  $K$
be a finite subset including the origin and suppose $A \subsetneq
  \Z^2\setminus K$. Then for $z\in A $, 
  \[  \sum_{w \in K}
    | H_{\Z^2\setminus K}^q(z,w) -
       H_A^q(z,w)|   \leq c \,\sqrt{ \frac{\diam(K)}R}
        \, \sum_{z' \in \p A \setminus
       K}
                   H_A(z,z'),\]
  where       $R=  \min\{|z'|: z' \in \p A \setminus K\}.$
  \end{lemma}
  
  \begin{proof}
  If $z\in A, w\in K$,
  \[    H_{\Z^2\setminus K}^q(z,w) =
       H_A^q(z,w) + \sum_{z' \in \p A \setminus K}
               H_A^q(z,z') \, H_{\Z^2\setminus K}^q(z',w).\]
  Therefore, by Lemma \ref{prop.kias},
  \begin{eqnarray*}
    \sum_{w \in  K}
    | H_{\Z^2\setminus K}^q(z,w) -
       H_A^q(z,w)| & \leq & \sum_{z' \in \p A \setminus
        K}  \sum_{w \in  K}  |H_A^q(z,z')|
        \, |H_{\Z^2\setminus K}^q(z',w)|\\
     & \leq & c \, \sqrt{\frac{\diam( K)}{R}} \, \sum_{z' \in \p A \setminus
        K}
                   H_A(z,z').
       \end{eqnarray*}      
    \end{proof}

Recall the stopping times for simple random walk $S_{n}$ on $\Z^{2}$ defined in the Introduction: \[\tau_+ = \min\{j \geq 0: S_j \in \Z_+\}\] and \[\sigma_R = \min\{j
   \geq 0 : |S_j| \geq R\}.\]

   We define the important function
   \[ v(z) =   \lim_{R \rightarrow \infty} R^{1/2}
    \, \Prob^z\{ \sigma_R < \tau_+\}, \]
    which, as we will see, gives a probabilistic definition of a discrete version of a constant times $\Im \sqrt{z}$. (See also \cite{GHP} for a related construction.)
    Of course, one has to check that the limit exists. This is most certainly known but we choose to give a self-contained proof
    which contains an error bound. This is done in Proposition \ref{limitexists}, the proof of which is given in the Appendix. We shall also need the following sequence of lemmas. The first gives a coupling of  Brownian motion and random walk given by the KMT approximation (see \cite{kmt1} and \cite{kmt2}) in a form which follows from Lemma A.5 and Theorem A.3 in \cite{BJK}:
 
\begin{lemma}\label{kmt}
There exists $c \in (0,\infty)$ so that for each $R < \infty$, there is a probability space on which a planar simple random walk started at $z$ and planar standard Brownian motion started at $z'$ with $|z-z'|\leq 2$ can be constructed in such a way that 
$$\Prob \left\{\sup_{0\leq t\leq T_R }|B_t-S_{2t}|>c\log R \right\}\leq cR^{-3},$$
where $T_R=\inf\{t\geq 0:|B_t|\geq R\}$.
\end{lemma}

We now show in a sequence of three lemmas that the probability that random walk exits the slit disk at a specific point, given that it exits on the circle, is about the same for all points on a mesoscopic scale. 
It turns out that the partial results in Lemmas \ref{lem:lawler-limic} and \ref{origintointeriorrw} are easier to prove in the slit square $%U^- = 
U_R^- = U_R \setminus [0, \ldots, R]$, where $U_R$ is as in \eqref{U}.  This is because the slit square can be split into a finite number of rectangles and the (discrete) Poisson kernel for a rectangle can be given explicitly; this idea is used in
Section 5 of \cite{BLV} and we state some of
the main results here.

\begin{lemma}\label{lem:lawler-limic}
If $z \in \p U_R^-$, let $d_R(z)$
be the distance from $z$ to the set $\{R, \pm R \pm iR\}$.
Then
\[   H_{\p U_R^{-}}(0,z) \asymp
 H_{\p U_R^{-}}(0,\p U_R)\,\frac{d_R(z)}
   {R^2}\asymp  
    \frac{1}{R^{1/2} }\, \frac{d_R(z)}
   {R^2}. \]
   More generally, if $m \leq R/2$
   and $w \in \p U_m,$
   \[  
 H_{  U_R^{- }}(w,z) \asymp  H_{  U_R^{- }}(w,\p U_R) \,\frac{d_R(z)}
   {R^2} \asymp  \frac{ |w-m|}{ m^{1/2}   
   R^{1/2}}
    \,   \frac{d_R(z)}
   {R^2}. \;\;\;\;  
     \] 
 
   \end{lemma}
   
 \begin{proof}  We sketch the proof since similar estimates have appeared in several places.
  For the proof that  $$ H_{  U_R^{- }}(w,\p U_R)
  \asymp \frac{|w-m|}{m^{1/2}R^{1/2}}$$ if $w \in \p U_m$, 
 see \cite[Proposition 5.3.2]{LL}. If $N =  \lfloor
 3R/4 \rfloor $, then we can write
 \[   H_{  U_R^{- }}(w,z) 
   = \sum_{\zeta,\xi \in \p U_N}
      H_{U_N^-}(w,\zeta) \, G_{U_R^-}(\zeta,\xi)
       \, H_{\p(  U_R^-\setminus \p U^-_{N})}
       (z, \xi),\]
 to see that
 \[    H_{  U_R^{- }}(w,z)  \asymp H_{U_N^-}(w,
 \p U_N ) \, H_{ \p(U_R^- \setminus \p U^-_{N})}
       (z, \p U_N).\]
This can be estimated using the fact that $H_{U_N^-}(w,
 \p U_N ) \asymp H_{U_R^-}(w,
 \p U_N )$ and by comparing the other term with the Poisson kernel in a
 rectangle which can be given
 explicitly in terms of a finite Fourier series (see \cite[Section 8.1]{LL}).
  
  \end{proof}
\begin{lemma}\label{origintointeriorrw}Let $0<\alpha<1$. There exists $\beta>0$ such that for $R\in\mathbb{N}, a \in U_{\lfloor R^{\alpha}\rfloor}^-$ and $b \in \p U_R$,
\[\frac{H_{\partial \slit_R}(0,b)}{H_{\partial \slit_R}(0,\p U_R)} = \frac{H_{\slit_R}(a ,b)}{H_{\slit_R}(a ,\p U_R)}(1+O(R^{-\beta})).\]
\end{lemma}

\def \F {{\mathcal F}}

\begin{proof} If $R^\alpha < m <R$, let $\mu_m,
\nu_m$ be the distribution of the first visit (after time
zero) to $\partial
U_m$ starting at $0,a$, respectively, conditioned
that the walk leaves $U_R^-$ at $\partial U_R$. Set 
\[   \|\mu_m - \nu_m\| =
 \frac 12 \sum_{w \in \p U_m} | \mu_m(w) - \nu_m(w)|.\]
Then a simple coupling argument implies that $\|\mu_m - \nu_m\|$ is decreasing in $m$
 and the last lemma shows that there exists $\delta >0$
 such that if $R^\alpha < m <R/4$, then
 \[    \|\mu_{2m} - \nu_{2m}\|
         \leq (1-\delta) \,  \|\mu_{ m} - \nu_{ m}\|.\]
  In particular,  if $R/4 \leq m < R/2$,
  \[    \|\mu_m - \nu_m\| \leq c \, R^{-\beta}, \]
  and
  \[  |\mu_R(b) -\nu_R(b)| 
    \leq   2 \, \|\mu_m - \nu_m\| \, \max_{w \in \p
    U_m} H_{U_R^-}(w,b)
      \leq c \, R^{-\beta} \,\mu_R(b).\]

\end{proof}
\begin{lemma}\label{conditionalboundary}
For every $0<\alpha < 1$, there exists $\beta>0$ such that for any $z,z'\in\Z^2$ with $|z|,|z'|\leq R^{\alpha}$, any $w\in C_R$,
$$\Prob^z\{S(\sigma_R)=w|\sigma_R < \tau_+\}=\Prob^{z'}\{S(\sigma_R)=w|\sigma_R < \tau_+\}(1+O(R^{-\beta})).$$
\end{lemma}

\begin{proof}Write $U, U^{-}$ for $U_{\lfloor R/2\rfloor}, U_{\lfloor R/2\rfloor}^-$ and
let $T$ be the first time that a simple random walk
leaves $U^-$.  We can write
\[  \Prob^z\{S(\sigma_R \wedge \tau_+) = w \mid S_T \in \p U\}
   = \sum_{b \in \p U} \frac{H_{U-}(z,b)}{H_{U^-} (z, \p U)}\, \Prob^b\{S(\sigma_R \wedge \tau_+) = w\},\]
and similarly for $z'$.  Using 
 Lemma \ref{origintointeriorrw} we see that for $|z|,|z'|\leq R^{\alpha}, b \in \p U$,
\[ \frac{H_{U^-}(z,b)}{H_{U^-} (z, \p U)}
 = \frac{H_{U^-}(z',b)}{H_{U^-} (z', \p U)}
  \,
(1+O(R^{-\beta})), \]
which  implies
\[  \Prob^z\{S(\sigma_R \wedge \tau_+) = w \mid S_T \in \p U\}
 =  \Prob^{z'}\{S(\sigma_R \wedge \tau_+) = w \mid S_T \in \p U\}\,
(1+O(R^{-\beta})), \]
\[   \Prob^z\{\sigma_R < \tau_+ \mid S_T
 \in \p U\} =  \Prob^{z'}\{\sigma_R < \tau_+\mid S_T
 \in \p U\}\,
(1+O(R^{-\beta})). \]
Since,
\[ \Prob^z\{S(\sigma_R)=w|\sigma_R < \tau_+\}
 =  \frac{\Prob^z\{S(\sigma_R) = w \mid S_T \in \p U\}}
   {\Prob^z\{\sigma_R  < \tau_+ \mid S_T \in \p U\}},\]
 and similarly for $z'$, we get the lemma.
\end{proof}
 
For $z \in \mathbb{C}$ we will write $\theta_z = \arg z \in [0, 2\pi)$.
 
\begin{lemma}\label{piover4} There exist $c < \infty, \beta > 0$
such that  
if $|z| \leq R^{3/4}$,  then
\[  
 \left|\E^z\left[\sin(\theta_{S(\sigma_R)}/2) \mid
   \sigma_R < \tau_+ \right] - \frac{\pi}{4}
    \right| \leq c \, R^{-\beta}. \]
\end{lemma}

    \begin{proof} 
Lemma~\ref{conditionalboundary} implies
    that there exists $\beta >0$ such that for $|z|,|w| \leq R^{3/4}$, %we have
 \[   \left|\E^z\left[\sin(\theta_{S(\sigma_R)}/2) \mid
   \sigma_R < \tau_+ \right]
   - \E^w\left[\sin(\theta_{S(\sigma_R)}/2) \mid
   \sigma_R < \tau_+ \right]\right| \leq c \, R^{-\beta} . \]
Hence it suffices to show the result for $z %= z_R 
= -\lfloor R^{3/4} \rfloor.$  
For the remainder of the proof, we write
$\Prob, \E$ for $\Prob^z,\E^z$ and
let $(B_t,S_t)$ be a Brownian motion and a simple random walk coupled as in Lemma \ref{kmt}, so that
\[   \Prob\left\{\sup_{0 \leq t \leq s_{2R}}
    |B_t - S_{2t}| \geq c_0 \, \log R\right\}
       \leq c_0 \, R^{-3}. \]
 Here  $s_{R}$ is the hitting time by  the
 Brownian motion of
  $\{|z| = R\}$. We let $T_{+}$ be the hitting time of $[0, \infty)$ by $B$.       
 
Using the Beurling estimate (see \cite{LLpaper} for the discrete version, \cite{graybook} for a discussion of the continuous case)
in an argument very similar to that in Proposition 3.1 in \cite{BJK}, one can show that
\begin{equation}\label{9.19}   \Prob\left\{\left|B_{s_R \wedge T_+} 
 - S_{\sigma_R \wedge \tau_+ }\right|
   \geq R^{1/2} \, \log R\right\}
       \leq c \, R^{-1/4}.
       \end{equation}
  Let $E_R, \tilde E_R$ be the events
    \[   E_R =  \{\sigma_R < \tau_+\}
    , \;\;\;\; \tilde E_R = \{ s_{R} < T_{+}\}. \]
Proposition 2.4.5 in \cite{greenbook} for the former and a direct calculation for the latter give
    $\Prob (E_R) \asymp \Prob (\tilde E_R) 
  \asymp R^{-1/8}$. 
 However,   
note that
\begin{eqnarray*}
\Prob (E_R \triangle \tilde E_R) &  \leq & \Prob\{d(B(s_R\wedge T_+),R)\leq R^{1/2}\log R\}\\
&&\hspace{5pc} +\Prob\left\{\left|B_{s_R \wedge T_+} 
 - S_{\sigma_R \wedge \tau_+ }\right|
   \geq R^{1/2} \, \log R\right\} \\
   & \leq & c\, R^{-1/2}\log R + c \, R^{-1/4} \leq c \, R^{-1/4},
\end{eqnarray*}
where the first term on the right of the first inequality can be estimated by a direct calculation using conformal invariance and the second is estimated in \eqref{9.19}. Here $\triangle$ denotes
    symmetric difference.  Hence
  we have produced a coupling of a Brownian
  motion conditioned on $\tilde E_R$ with
  a random walk conditioned on $E_R$ such that,
  except for an event of probability $O(R^{-1/8})$,
 $ |B_{s_R \wedge T_+} 
 - S_{\sigma_R \wedge \tau_+ } |
   \leq R^{1/2} \, \log R$. 
In particular,
\[ \left| \E  \left[\sin(\theta_{S(\sigma_R)}/2) \mid
  E_R \right]
    - 
  \E  \left[\sin(\theta_{B(s_R)}/2) \mid
 \tilde E_R \right] \right| \leq c \, R^{-1/8}. \]
We are left with estimating the quantity for Brownian
motion which can be done using conformal invariance. 
    The Poisson kernel in the slit disk is (see for instance the proof of Lemma 5.1 in \cite{BLV})
    \begin{equation}\label{poissondisk}
    h_{\Disk\setminus [0,1)}(\ee e^{i\mu}, e^{i\theta})=\frac{1}{2\pi}\ee^{1/2}\sin(\mu/2)\sin(\theta/2) +O(\ee).
    \end{equation}
    This implies that, with $\xi = s_1\wedge T_+$, 
    $$\Prob^{\ee e^{i\mu}}\{|B_\xi|=1\}=\frac{2}{\pi}\ee^{1/2}\sin(\mu/2) +O(\ee),$$
    so the Poisson kernel in the slit disk conditional on leaving at the circle before hitting the positive real line is 
    \begin{equation}\label{Poisslit}
    \bar{h}_{\Disk\setminus [0,1)}(\ee e^{i\mu}, e^{i\theta})=\frac{1}{4}\sin(\theta/2) +O(\ee^{1/2}).
    \end{equation}
   Therefore, with $\xi$ the first exit time of $\Disk\setminus [0,1)$,
    \begin{equation}\label{estimateforB}
  \E^{\ee e^{i\mu}}\left[\sin\left(\frac{\theta_{B_{\xi}}}{2}\right)\bigg||B_{\xi}|=1\right]=\frac{1}{4}\int_0^{2\pi}(\sin^2(\theta/2)+O(\ee^{1/2})\sin(\theta/2))\,d\theta,
  \end{equation}
  and the right-hand side equals $\frac{\pi}{4}+O(\ee^{1/2})$,
from which our estimate follows by scaling.

    \end{proof}
    
We postpone the proof of the following proposition to the Appendix.

\begin{proposition}\label{limitexists}
    For each $z$, the limit
\[   v(z) = \lim_{R \rightarrow \infty} R^{1/2}
  \, \Prob^z\{\sigma_R < \tau_+\} \]
  exists.  Moreover, there exists $c < \infty$
  such that for all $z \in \Z^{2} \setminus \{0\}$
  \[       \left| v(z) - \frac 4 \pi \, \Im \, \sqrt z  \right|
     \leq   c\, \frac{\sin(\theta_{z}/2)}{|z|^{1/2}}.\]
 \end{proposition}
  Note that 
  \[  v(z) = 0 , \;\;\;\; z \in  \Z_+, \]
  \[  \dlaplace  v(z) = 0 , \;\;\;
   z \in \Z^2 \setminus \Z_+, \]
 \[\dlaplace^q v(z) = 0, \;\;\;\; z \in  \Z^2\setminus\{0\}.\]
 Defining
  \[       f(z) = \Im \, \sqrt z =   |z|^{1/2} \, \sin(\theta_z/2),
   \]
 it follows from \eqref{fR} below 
 %and Lemma \ref{piover4} 
 that
 \begin{eqnarray*}
 v(z) & = & \lim_{R \rightarrow \infty} \frac{4}\pi
 \,\E^z[f(S_{\sigma_R \wedge \tau_+})]\\
 & = &  \lim_{R \rightarrow \infty}
    \frac{4}\pi \sum_{w \in   C_R}
         H_{D_R}^q(z,w) \, f(w) .
  \end{eqnarray*}

  The second equality uses the fact that
  $f(w) = f(\overline{w})$ and hence  if $z
   \in D_R \cap \Z_+, $  we can reflect paths in $\mathbb{R}$ to see that
   \[        \sum_{w \in   C_R}
         H_{D_R}^q(z,w) \, f(w)  = 0.\]

   We now consider a discrete harmonic conjugate of $v$ (though note that $u$ lives on the same lattice as $v$), 
   \begin{equation}\label{v}    
   u(x+iy) =  \left\{ 
       \begin{array} {ll}  v(-x + iy),  &  y \geq   0\\
           - v(-x + iy) , & y < 0 \end{array} \right.
   \end{equation}
           We can think of $u$ as (a constant times) a discrete version of $\Re \, \sqrt{z}$.
   Then,
   \[ u(z) = 0 , \;\;\;\; z \in  \Z_-, \]
  \[  \dlaplace u(z) = 0 , \;\;\;
   z \in \Z^2 \setminus \{(x,y)\in\Z^2:x\geq 0, y\in \{-1,0\}\}, \]
 \[    \dlaplace ^qu(z) = 0, \;\;\;\; z \in  \Z^2\setminus\{0\}.\]

  If $\eta \in \saws$, we define
   \begin{equation}\label{ueta}   v_\eta(z) = v(z) - \sum_{y \in \eta} H^q_{\Z^2\setminus\eta}\
   (z,y) \, v(y) 
     \end{equation}
   \begin{equation}\label{veta}   
   u_\eta(z) = u(z) -\sum_{y \in \eta} H^q_{\Z^2\setminus\eta}
   (z,y)\, u(y)  . 
   \end{equation}
    Then we have
    \[    u_\eta (z) = v_\eta(z) = 0 , \;\;\;\ z \in \eta,\]
    \[  \dlaplace ^qu_\eta(z) = \dlaplace ^q v_\eta(z) = 0 , \;\;\;
        z \in \Z^2\setminus\eta.\]
        %As in the proof of 
        By Lemma~\ref{prop.kias}, if $|z| > n$, 
        \begin{eqnarray*}
        |\sum_{y \in \eta} H^q_{\Z^2\setminus\eta}(z,y) \, v(y) | & = &|\sum_{w \in D_n} \sum_{y \in \eta} H^q_{
        \Z^2 \setminus D_n}(z,w)H^q_{\Z^2\setminus\eta}(w,y) \, v(y) | \\
        &\le& \sum_{w \in  D_n} \sum_{y \in \eta} 
        |H^q_{ \Z^2 \setminus D_n}(z,w)||H^q_{\Z^2\setminus\eta}(w,y) \, v(y) |\\
        & \le & c \max_{y \in \eta} |v(y)| \sqrt{\frac{n}{|z|}},
        \end{eqnarray*}
        and similarly for $u$.
      Hence, if $|z|$ is large enough,
     \begin{eqnarray*}  
     |v_\eta(z) - v(z)| + |u_\eta(z) - u(z)| & = & |\sum_{y \in \eta} H^q_{\Z^2\setminus\eta}(z,y) \, v(y) |  + |\sum_{y \in \eta} H^q_{\Z^2\setminus\eta}(z,y) \, u(y) |\\
         & \le & c_\eta \, |z|^{-1/2},
        \end{eqnarray*}
        where $c_\eta$ depends only on $\eta$.
                
    Recall that $A_n$ is the infinite strip
    $\{x+iy: |x| < n\}$.
     We define, for $z\in A_n$,  
   \[      v_\eta^{(n)} (z) = H^q_{A_n \setminus\eta}(z,-n), \;\;\;\;
       u_\eta^{(n)}(z) = H^q_{A_n \setminus \eta}(z,n) , \]
       where $u_\eta^{(n)} \equiv v_\eta^{(n)} \equiv 0$
       for  $z\in \eta$. 
    Finally, let
    \[     D^q_n (\eta) = \det \left[\begin{array}{cc}
                 L^qv_\eta^{(n)}(\eta_-) & L^qv_\eta^{(n)}(\eta_+)\\
                    L^q u_\eta^{(n)}(\eta_-) & L^q u_\eta^{(n)}(\eta_+)
                      \end{array} \right].\]

Before proving our main result for this section, Proposition~\ref{main}, we still need a convergence result (Lemma \ref{lem:jul26}) which will rely on Lemma \ref{limitpartial} for which the following definitions are needed:

For $z \in D_m, w \in C_m$, we let
  \[      h^\eta_m(z,w)=   \frac12 \, [H^q_{D_m \setminus \eta}(z,w)
 + H^q_{D_m \setminus \eta}(z,\overline w)].\]
 Note that for $z\in\eta, h^\eta_m(z,w)=0$.   We can think of $\{0\}$ as the ``empty'' SAW and define
    \[h_m(z,w) = h_m^{\{0\}}(z,w)= \frac12 \, [H^q_{D_m \setminus \{0\}}(z,w)
 + H^q_{D_m \setminus \{0\} }(z,\overline w)].\]
 
Note that $h_m(z,w) = 0 $ if $z \in \Z_+$,
so %by the harmonicity we have that
for all $z\in D_m$, 
 \begin{equation}\label{071219}  
 h_m(z,w)=\frac 12 \,\left[
 \Prob^z\{S(\sigma_m\wedge\tau_+)
  = w\}+\Prob^z\{S(\sigma_m \wedge\tau_+)=\overline w\}\right].
  \end{equation}
%We generalize this definition and for $z \in D_m, w \in C_m$ we let

Notice that for any function $\phi$ defined on $C_m$, $\sum_{w \in C_m}\phi(w)=\sum_{w \in  C_m}\phi(\bar{w})$, so
  if  $\phi(w) = \phi(\overline w)$, then
 \begin{equation}\label{07-26}
       \sum_{w \in  C_m}
      H^q_{D_m \setminus \eta}(z,w) \, \phi(w)
        = \sum_{w \in  C_m}   h_m^\eta(z,w) \, \phi(w).
 \end{equation}

\begin{lemma}\label{limitpartial}
There exist  $0 <u, c < \infty$, 
such that for every SAW $\eta$ through $0$ and every  $z \in \Z^2\setminus\eta$,  if  $r = \max\{|z|, \diam(\eta)\}  $,
 $m \geq 2r$, and $w\in C_m$,
 \[       |h_m^\eta(z,w) - v_\eta(z) \, m^{-1/2} \, \mu_m(w)|
      \leq c \, (m/r)^{-\frac {1+u}{2} } \, \mu_m(w),\]
      where  
   \[        \mu_m(w) =  \frac{h_m(-1,w)}{\Prob^{-1}
      \{\sigma_m < \tau_+\}}.\]
              \end{lemma}
    
\begin{proof}
      We 
   define
   $a$ by $m = a^2 r$.  
Using \eqref{071219} and coupling $h$-processes started from different points as in Lemma \ref{conditionalboundary}, we can see that there exists $0 < u < 1/2$
    such that  for any $x \in D_{ar } \cup C_{ar }, w \in C_m$,
\begin{equation}\label{htomu}
  h_m(x,w) =
      \Prob^x\{\sigma_m < \tau_+\} \, 
        \mu_m(w) \ [1 + O(a^{-u})],
\end{equation}
   where     
\begin{equation}\label{mu_m}       \mu_m(w) =  \frac{h_m(-1,w)}{\Prob^{-1}
      \{\sigma_m < \tau_+\}} \asymp m^{-1} \, 
          \sin(\theta_w/2).
\end{equation}
 Also, if $x \in C_{ar},$ using $m=a^2 r$ and $\diam(\eta) \le r$,
\begin{eqnarray}\label{htoheta} \notag
   |h_m(x,w) - h_m^\eta(x,w)|
    & \leq & \sum_{y \in \eta}
       |H^q_{ D_{m} \setminus\eta } 
         (x,y)| \, |h_m(y,w)| \\ \notag
    &  \leq & c \, a^{-1 } \, \mu_m(w)
      \, \sum_{y \in \eta}
     |H^q_{ D_{m} \setminus\eta} 
         (x,y)|\\&  \leq & c 
        \, a^{-3/2 } \, \mu_m(w),
  \end{eqnarray}
  where we used \eqref{htomu} and the Beurling estimate for the second inequality and Lemmas \ref{prop.kias} and \ref{withoutAn} in the last inequality.
Since for $x \in C_{ar}, \Prob^x\{\sigma_m < \tau_+\}
\asymp a^{-1/2} \,  \sin(\theta_x/2)$,
we get from \eqref{htomu} and \eqref{htoheta}
\[   |h_m^\eta(x,w) -  \Prob^x\{\sigma_m < \tau_+\} \, 
        \mu_m(w) | \leq c \,a^{-1/2}\, \mu_m(w)
        \, [a^{-1}+ \sin(\theta_x/2) a^{-u}    
          ]. 
         \]
By the strong Markov property (using that $h_m^\eta(\cdot, w) \equiv 0$ on $\eta$), we can write 
\begin{eqnarray*}
h_m^\eta(z,w)
  & = & \sum_{x \in C_{ar}}
     H_{ D_{ar} \setminus\eta }^q(z,x) \, h_m^\eta
     (x,w), 
\end{eqnarray*}
and hence
\begin{eqnarray*}
 \lefteqn{\left|h_m^\eta(z,w)
    - \sum_{x \in C_{ar}} 
   H_{ D_{ar} \setminus\eta }^q(z,x) \,
    \Prob^x\{\sigma_m < \tau_+\} \, 
        \mu_m(w)   \right|}\\
        & \leq & c
    \sum_{x \in C_{ar}}  |H_{ D_{ar} \setminus\eta }^q(z,x)|
     \,  \mu_m(w)
        \, a^{-1/2}\,a^{-u}\\
        & \leq & c\, (m/r)^{-\frac {1+u}{2}}
          \, \mu_m(w),
\end{eqnarray*}   
where the last inequality follows from Lemma \ref{qBeurling}. This gives   
\begin{equation}
\label{aug29.1}
  \left|h_m^\eta(z,w)
  -
      \kappa_m(\eta,z) \, m^{-1/2}\,  \mu_m(w)   \right|
      \leq  c\, (m/r)^{-\frac {1+u}{2}}
          \, \mu_m(w),
          \end{equation}
   where
 \[   \kappa_m(\eta,z) =   m^{1/2}\sum_{x \in C_{ar}} 
   H_{D_{ar} \setminus\eta }^q(z,x) \,
    \Prob^x\{\sigma_m < \tau_+\}   .\]
We want to show that this is close to $v_\eta(z)$. Note that by Lemma \ref{qBeurling},
\begin{equation}\label{kappabound}
 |\kappa_m(\eta,z) | \leq m^{1/2}\max_{x \in C_{ar}}
 \Prob^x\{\sigma_m < \tau_+\}  
\sum_{x \in C_{ar} }
   |H_{ D_{ar} \setminus\eta }^q(z,x) |
     \leq O(r^{1/2}).
     \end{equation}

Using \eqref{mu_m} and \eqref{aug29.1}, we see that
\begin{equation}     \label{09-06b} 
 \left|   \sum_{w \in C_m} [ h_m^{\eta}  
       (z,w)
        - \kappa_m(\eta,z)  m^{-1/2}
         \mu_m(w)]  \,   \frac 4 \pi \,  m^{1/2}  \, \sin(\theta_w/2)\right|  \leq c \, m^{-u}r^{(1+u)/2}.
         \end{equation}
 But we know from Lemma \ref{piover4} that
 \[   
 \sum_{w \in C_m}
        \mu_m(w) \, \sin(\theta_w/2) =
            \frac \pi 4 + O(m^{-u}), \]
so by \eqref{kappabound}

\begin{equation}     \label{09-06c} 
 \left|   \sum_{w \in C_m}  h_m^{\eta}  
       (z,w) \frac 4 \pi \, m^{1/2}  \, \sin(\theta_w/2)
        - \kappa_m(\eta,z) 
        \right|  \leq c \, m^{-u/2}r^{(1+u)/2}.
         \end{equation}
         
For $w \in C_m$, it follows from Lemma \ref{prop.kias} and the boundedness of $v$ implied by Proposition \ref{limitexists} that
\begin{equation}\label{v_eta} v_\eta(w) = v(w) + O((r/m)^{1/2})
   =   \frac 4 \pi \, \sqrt m  \, \sin(\theta_w/2)
     + O((r/m)^{1/2}).
\end{equation}

Since by the strong Markov property (using that $v_\eta \equiv 0$ on $\eta$),
\[   v_\eta(z) =  \sum_{w \in C_m}  H_{  D_m \setminus\eta }^q
       (z,w) \,  v_\eta(w),\]
equations \eqref{07-26} and \eqref{v_eta}, together with Lemma \ref{qBeurling}, imply that

\begin{eqnarray}\label{09-06a} \notag
\lefteqn{\left| v_\eta(z) - \sum_{w \in C_m}  h_m^{\eta}  
       (z,w) \,   \frac 4 \pi \, \sqrt m  \, \sin(\theta_w/2)\right|}\\ \notag
    & = & 
    \left| v_\eta(z) - \sum_{w \in C_m}  H_{  D_m \setminus\eta }^q
       (z,w) \,  \frac 4 \pi \, \sqrt m  \, \sin(\theta_w/2)\right|\\ \notag
       & \leq & O((r/m)^{1/2})
          \sum_{w \in C_m} |H_{ D_m \setminus\eta }^q
       (z,w)|  \\
       & = &  O(r/m).
       \end{eqnarray}
 Therefore, by \eqref{09-06c} and \eqref{09-06a}, since $0<u<1/2$,
\begin{equation*}\label{kappalimit}
    v_\eta(z) = \kappa_m(\eta,z)  + O(m^{-u/2}r^{(1+u)/2}),
    \end{equation*}
 which shows that $\lim_{m\to \infty} \kappa_m(\eta,z) = v_\eta(z)$
and \eqref{aug29.1} becomes 
\[  \left|h_m^\eta(z,w)
  -
      v_\eta(z)   \, m^{-1/2}\,  \mu_m(w)   \right|
      \leq  c\, (m/r)^{-\frac {1+u}{2}}
          \, \mu_m(w). \]

  \end{proof}

   \begin{lemma}\label{lem:jul26} There exists $0 < c_0 < \infty$ such that 
   for all $z \in \Z^2$, 
  \[  u_\eta(z) = c_0 \lim_{n \rightarrow \infty}
         n^{3/2}\,  u_\eta^{(n)}(z) ,\]
         \[  v_\eta(z) = c_0 \lim_{n \rightarrow \infty}
         n^{3/2}\, v_\eta^{(n)}(z) .\]
         \end{lemma}

     \begin{proof}    We will prove the second limit; the first
   is done similarly.   Let $m = n/2$ and write
   \[   v_\eta^{(n)}(z)
    = \sum_{w \in C_m}
      H_{D_m \setminus\eta }^q(z,w) \,  H_{A_n \setminus\eta }^q(w,-n).\]
For $w \in C_m$,
\begin{eqnarray*}   
H_{A_n \setminus \eta}^q(w,-n)
&=& H^q_{A_n}(w,-n) - \sum_{y \in \eta} H_{A_n \setminus\eta}^q(w,y) 
\, H_{A_n}^q(y,-n)\\
& =& H^q_{A_n}(w,-n) + O(n^{-2}),
\end{eqnarray*}
where the second equality follows from Lemmas \ref{prop.kias}, \ref{withoutAn}, the Beurling estimate, and the fact that the discrete Poisson kernel in the half-plane is a discrete version of the Cauchy distribution (see, e.g., Lemma 4.2.1 in \cite{benesthesis}).
 Therefore, by Lemma \ref{qBeurling} and \eqref{07-26},  
 \begin{eqnarray*}   
  v_\eta^{(n)}(z) + O(m^{-5/2})
 & = & 
      \sum_{w \in C_m}
      H_{D_m  \setminus\eta}^q(z,w) \,  H_{A_n}^q(w,-n)\\
    &   = &\sum_{w \in C_m}
      h_m^\eta(z,w) \,  H_{A_n}^q(w,-n).
      \end{eqnarray*}
Using the Lemma \ref{limitpartial}, we see that
\[\sum_{w \in C_m}
      h_m^\eta(z,w) \,  H_{A_n}^q(w,-n)
      = O(m^{-\frac{3+u}{2})}
 +  v_\eta(z)
  \sum_{w \in C_m}
      \frac{\mu_m(w) \,  H_{A_n}^q(w,-n)}{m^{1/2}}.\]

Note that symmetry and the argument of Lemma \ref{piover4} imply that
\begin{eqnarray*}
 &&\sum_{w \in C_m}
 \mu_m(w)  H_{A_n}^q(w,-n)\\
 & & \hspace{3pc}= \sum_{w \in C_m}
\frac{\Prob^{-1}(S(\sigma_m\wedge \tau_+)=w)}{\Prob^{-1}(\sigma_m<\tau_+)} \,  H_{A_n^-}(w,-n)\\
 & & \hspace{3pc} = \sum_{w \in C_m, d(w,m)\geq m^{3/4}}
\frac{\Prob^{-1}(S(\sigma_m\wedge \tau_+)=w)}{4\Prob^{-1}(\sigma_m<\tau_+)}  \\
& &\hspace{14pc} \times G_{A_n^-}(w,-(n-1))(1+O(n^{-1/2})).
 \end{eqnarray*}
The argument of Lemma \ref{piover4} and Theorem 8.1 in \cite{BJK} imply that 
$$ \mu_m(w) \,  H_{A_n}^q(w,-n)=\frac{1}{2\pi}\int\bar{h}(-1,w)g_{A_n^-}(w,-(n-1))(1+O(n^{-1/2+\epsilon}))\,dw,$$
where $\bar{h}$ is the Brownian Poisson kernel in the slit disk conditional on not leaving at the slit and the integral is over $\{w:|w|=m, d(w,m)\geq R^{3/4}\}$. This last expression can be shown to equal $\frac{c'}{n} \, [1+ O(n^{-u})]$ for some $c'$.

 \end{proof}
  
\subsection{Proof of Proposition~\ref{main} and an example}
   
   \begin{proof}[Proof of Proposition~\ref{main}] Recall that we want to check that
   \[  \hat{p}(\eta^\zeta \mid \eta) = \frac{\hat{p}(\eta^\zeta)}{\hat p (\eta)}= \frac 14 \, G_{\Z^2\setminus\eta}^q(\zeta,\zeta)
\, \left|\frac{D^q(\eta^\zeta)}{D^q(\eta)}\right|,\]
where
\[ D^q(\eta) = \det \left[\begin{array}{cc}
                 L^qv_\eta(\eta_-) & L^qv_\eta(\eta_+)\\
                    L^q u_\eta(\eta_-) & L^q u_\eta(\eta_+)
                      \end{array} \right].\]
   First note that Lemma~\ref{lem:jul26} implies
   \[
   \lim_{n\to \infty}\frac{D^q_n(\eta^\zeta)}{D^q_n(\eta)} = \frac{D^q(\eta^\zeta)}{D^q(\eta)}.
   \]
   We will use Theorem 3.1 of \cite{BLV} applied to
       $(A_n,-n,n)$.  We state the result here using the
    notation we use in the present paper with $\tilde \eta := \eta^\zeta$.  This exact
    expression is obtained by considering the loop measure
   description of LERW with the signed measure $q$ and making
   use of an identity due to Fomin; see Section~3 of \cite{BLV} for more details. We have
   \[
   \frac{\hat p^{(n)}(\tilde \eta)
       + \hat p^{(n)}(\tilde \eta^R) }       
    {  \hat p^{(n)} (\eta) + \hat p^{(n)} (  \eta^R)}  
     = \frac{1}{4} \, G^q_{A_n \setminus \eta}(\zeta,\zeta) \,
       \left|\frac{D^q_n({\tilde \eta})}{D^q_n( \eta)}\right|,  \]
       where $\eta^R$ is the reversed path.
  A similar expression using the loop measure and Fomin's identity
  with the original probability $p$ allows one to conclude
$ \hat p^{(n)} (\eta) \approx    \hat p^{(n)} (  \eta^R)$ and similarly
for $\tilde \eta$.  More precisely,
\[             \hat p^{(n)} (\eta) =  \hat p^{(n)} (  \eta^R)
\, \left[1 + O_\eta(n^{-\beta}) \right]. \]
See Section~3 of \cite{BLV} for details.
Therefore,
\[  \lim_{n \rightarrow \infty} 
  \frac{\hat p^{(n)}(\tilde \eta)}
      {\hat p^{(n)}(  \eta ) }  = \lim_ {n \rightarrow \infty}
          4^{-1} \, G^q_{A_n \setminus \eta}(\zeta,\zeta) \,
       \left|\frac{D^q_n({\tilde \eta})}{D^q_n( \eta)}\right| =
        4^{-1} G^q_{\Z^2 \setminus \eta}(\zeta,\zeta)
         \,   \left|\frac{D^q ({\tilde \eta})}{D^q ( \eta)}\right|.\]
              
 \end{proof}  

We conclude this paper with an explicit result which Kenyon and Wilson also obtained through a different argument in \cite{KenWil}:
 \begin{example}Although $\hat{p}(\eta)$ is hard to compute in general, one nice
 case is the straight line when $d=2$, $\eta = \eta_k  = [0,1,2\ldots,k],\,  w=k+1$, so that $\eta^w= \eta_{k+1}$.  We first note that
 \[      F^q_{\Z^2\setminus\eta}(k+1,k+1) = F_{\Z^2 \setminus \{\ldots,k-1,k\}}(k+1,k+1) =
                F_{\Z_-}(1,1) = 4(\sqrt 2 - 1) . \]
  The first equality uses symmetry --- if we consider loops that hit the negative real axis, the total $q$ weight is zero since positive loops (those come from the positive $y$-axis) cancel with negative loops. The last equality takes work but is known, see \cite[Proposition 9.9.8]{LL}.  Also, symmetry shows
  that
  \[    \dlaplace^qv_\eta^{(n)}(0) =  \dlaplace^q u_{\eta}^{(n)}(k) = 0 , \]
   \[    \dlaplace^q v_\eta^{(n)}(0) = H_{\p(A_n \setminus \Z_+)}(0,-n) ,\]
   \[     \dlaplace^q u_\eta^{(n)}(k) = 
   H_{\p(A_n \setminus \{\ldots,k-1,k\})}(k,n) , \]
and by taking the limit as $n \rightarrow \infty$, we see that
\[       D^q( \eta^w) = D^q(\eta), \;\;\;\; \hat p( \eta^w \mid \eta)
 = \sqrt  2 -1.
 \]
 Hence,
 \[
    \hat p(\eta_k) = \frac{1}{4} \, (\sqrt 2 - 1)^{k-1}. \]
    \qed
  \end{example}

\section*{Appendix: Proof of Proposition \ref{limitexists}}
In this appendix, we prove Proposition \ref{limitexists}, which we restate now for convenience:

\newtheorem*{limitexists}{Proposition \ref{limitexists}}
\begin{limitexists}

%\begin{proposition}\label{limitexists}
    For each $z$, the limit
\[   v(z) := \lim_{R \rightarrow \infty} R^{1/2}
  \, \Prob^z\{\sigma_R < \tau_+\} \]
  exists.  Moreover, there exists $c < \infty$
  such that for all $z \in \Z^{2} \setminus \{0\}$
  \[       \left| v(z) - \frac 4 \pi \, \Im \, \sqrt z  \right|
     \leq   c\, \frac{\sin(\theta_{z}/2)}{|z|^{1/2}}.\]
% \end{proposition}
 \end{limitexists}
 \begin{proof}  Define
   \[       f(z) = \Im \sqrt z =   |z|^{1/2} \, \sin(\theta_z/2)
   .\]
We write $\Delta$ for the continuous and
 $L$ for the discrete Laplacian.  We also write
 $A = \Z^2 \setminus \Z_+$.  Since $f$ is the imaginary part
 of a holomorphic function on $\C \setminus [0,\infty)$,
  $\Delta f(z) = 0$ for $z \in 
 A$.
Let
\[ 
    f_R(z) =   \E^z\left[f(S(\sigma_R \wedge \tau_+)) \right].\]
    Note that
\begin{eqnarray}\label{fR} \notag
 f_R(z)
& = & [R^{1/2}  +O(1)]\, \E^z\left[\sin(\theta_{S(\sigma_R)}/2) 1\{
   \sigma_R < \tau_+\} \right]\\ \notag
   & = & [R^{1/2}  +O(1)] \, \Prob^z\{ \sigma_R < \tau_+\}
    \, \E^z\left[\sin(\theta_{S(\sigma_R)}/2) \mid
   \sigma_R < \tau_+ \right]\\ 
   &= & R^{1/2}\, \Prob^z\{ \sigma_R < \tau_+\}
    \, \frac{\pi}{4} (1+O(R^{-(\beta\wedge 1/2)})),
\end{eqnarray}
where the last estimate follows from Lemma \ref{piover4}.  
 The main estimate we will prove is the following:
  There exists $c < \infty$ such that for all
 $z$ and all $R > |z|$, 
 \begin{equation}  \label{mainclaim}
    |f(z) - f_R(z)| \leq c \, f(z) \, |z|^{-1}.
\end{equation}

This estimate will imply the lemma as we now show. Suppose $R < R'$. Since $z \mapsto f_R(z) - f_{R'}(z)$ is discrete harmonic in $D_R\setminus \Z_+$ and equals $0$ on the slit and $f\equiv f_R$ on $C_R$ (see \eqref{discretecircle}), we have
\[
|f_R(z) - f_{R'}(z)| \le \sum_{w \in C_R}H_{C_R}(z,w)|f(w) - f_{R'}(w)|. 
\]
Using \eqref{mainclaim}, 
\[
 \sum_{w \in C_R}H_{C_R}(z,w)|f(w) - f_{R'}(w)| \le  \sum_{w \in C_R}H_{C_R}(z,w)|f(w)|R^{-1} \le c \, R^{-1/2}.
\]
It follows that $\lim_{R \rightarrow \infty}  f_R(z)$ exists and, by \eqref{fR}, so does the limit defining $v$ and we have
\[       \lim_{R \rightarrow \infty}  f_R(z) = \frac{\pi}{4} v(z). \]
%In addition, 
% \[    |f(z) - \frac{\pi}{4}v(z)| \leq c \, f(z) \, |z|^{-1} .\]
%Indeed, \[   \frac{\pi}{4}v(z) = f(z) + \E^z \left[\sum_{j=0}^{\tau_+ - 1}
%       Lf(S_j) \right]\]
%       and the sum in this expression is estimated by \eqref{jan10.2} below.
The proof of the proposition will therefore be complete once we prove \eqref{mainclaim}, which we will now do.

If $f$ were discrete harmonic on $A$, then we would have
$f(z) = f_R(z)$.  The key estimate is obtained by comparing the
discrete and continuous Laplacian.
  First, notice that there
 exists $c < \infty$ such that
 \begin{equation}  \label{jan10.1}
       |Lf(z)| \leq \frac{c}{|z|^{7/2}}, \;\;\;\;
  z \in A.
  \end{equation}
 Indeed, $f$ is (continuous) harmonic and since the discrete Laplacian $L$ is a finite difference approximation to the continuous Laplacian $\Delta$, there is a universal $c$ such that
 \[    \left|   Lf(z)  - \frac{1}{2d} \,  \Delta f(z)  \right| \le c  M(z), \]
where $M(z)$ is the maximum value of
 the \emph{fourth} derivatives of $f$
  in the disk of radius $1$
 about $z$. See Section 1.5 in \cite{LL} for more details. 
 One could compute the fourth derivatives directly but it
 is easier to 
 use the bound obtained by differentiating the Poisson formula.  If $\sin(\theta_z/2)$ is small, then
 we can use the reflected harmonic function about zero (that is,
 changing the sign to negative in the negative imaginary
 half plane).  This gives a harmonic function on the ball
 of radius $|z|/2$ about $z$ whose maximum value
 is $O(|z|^{1/2})$ and hence whose fourth derivatives at $z$
 are $O(|z|^{-7/2})$.

 We now claim that there exists $c$ such that for all $z$,
 \begin{equation}  \label{jan10.2}
      \sum_{w \in A}  G_A(z,w) \, |Lf(w)| \leq c \,f(z) \, |z|^{-1}.
      \end{equation}
 To see this, we note that if $|z| \leq R$,
  \begin{equation}  \label{jun13.1}
   \sum_{R \leq |w| \leq 2R}  G_A(z,w) \leq c \,\sin(\theta_z/2) \, \sqrt{ \frac{|z|}{R}}
  \, R^2, 
  \end{equation}
which can be shown as follows:  
Let  
 \[       b_R(z) =   \sum_{R \leq |w| \leq 2R}  G_A(z,w), \]
 and write $ b_R = \sup_{|z| \leq 4R}  b_R(z)$. It is easy to verify that there exists $c>0$ such that for $k\in\mathbb{N}, |z|\leq 4R$, 
 $$\Prob^z\{\sigma_{4R}\geq k\}\leq \exp\{-ck/R^2\},$$
 so
there exists $c_1 < \infty$ such that 
for all $|z| \leq 4R$,
$$\E^z[\sigma_{4R}]=\sum_{k\geq 1}\Prob^z(\sigma_{4R}\geq k)
 \leq c_1\, R^2.$$   
 Also there exists $\rho< 1$ such that
 starting at $ z \in \p C_{4R}$, the probability to reach
 $C_{2R}$ without leaving $A$ is at most $\rho$.  
 Therefore, 
  \[  b_R   \leq c_1 \, R^2 +  \rho \, b_R, \]
  which gives $b_R \leq c_1 R^2/(1-\rho).$
  More generally, using gambler's ruin estimates for one-dimensional walks (see the argument for \eqref{BeurlingSharper} just below),
  we see that
  \[       \E^z[\sigma_{4R}] \leq c \, \sin(\theta_z/2) \, R^2.\]
  If $R/2 \leq |z| \leq R$, we then get
  \[       b_R(z) \leq  c \, \sin(\theta_z/2) \, R^2
       + \Prob^z\{\sigma_R < \tau_+\}\, b_R
          \leq   c \, \sin(\theta_z/2) \, R^2
       + c \, \sin(\theta_z/2) \, \rho \, b_R.\]
    Finally for $|z| \leq R/2$, we get
    \[   b_R(z)  \leq \Prob^z\{\sigma_R < \tau_+\} \, b_R.\]
It follows from Theorem 1 in \cite{LLpaper} that
\begin{equation}\label{LLBeurling}       \Prob^z\{\sigma_R < \tau_+\} \leq
     c \, \sqrt{\frac{|z|}{R}}.
 \end{equation}
It takes little work to show the following stronger estimate:
\begin{equation}\label{BeurlingSharper}       \Prob^z\{\sigma_R < \tau_+\} \leq
     c \,  \sin(\theta_z/2) \, \sqrt{\frac{|z|}{R}}.
 \end{equation}
      
 Indeed, it suffices to consider the case $\theta_z\leq \pi/4$. If $|z|\leq R/2$, then $\sigma_R < \tau_+$ implies that the random walk leaves the square with corners $(\Re \,z/2,0), (3\Re \,z/2,0), (\Re \,z/2,\Re \,z)$, and $(3\Re \,z/2,\Re \,z)$ at a point with nonzero imaginary part. Section 8.1.2 in \cite{LL} implies that this happens with probability comparable to $\sin(\theta_z)$, which is comparable to $\sin(\theta_z/2)$. Any such point has magnitude comparable to that of $z$, so the strong Markov property together with \eqref{LLBeurling} imply \eqref{BeurlingSharper}.
Note that we assumed for simplicity that $\Re \, z/4\in\mathbb{N}$, but the argument is the same without that assumption. 
 
\noindent   This gives \eqref{jun13.1}.  Using \eqref {jan10.1},
   we see that 
   for $|z| \leq R$,
  \[    \sum_{R \leq |w| \leq 2R}  G_A(z,w)\, |Lf(w)| \leq c \,\sin(\theta_z/2) \, \sqrt{ \frac{|z|}{R}}
  \, R^{-3/2} .\]
 In particular, if $n  \geq 0$.
    \begin{eqnarray*}
       \sum_{2^{n-1} |z|  \leq |w| \leq 2^n|z|}  G_A(z,w) \, |Lf(w)| & \leq & c \,\sin(\theta_z/2) \, 2^{-n/2} \, 
   2^{-3n/2} \, |z|^{-3/2}\\
   & \leq &  c\, 2^{-2n}  \, f(z) \, |z|^{-1} 
   \end{eqnarray*}  
  By summing over  all $n \geq 1$, we see that
\begin{equation}  \label{jun13.3}
  \sum_{|w| \geq |z|} G_A(z,w) \, |Lf(w)|
      \leq c \, f(z) \, |z|^{-1}.
      \end{equation}
Similarly, if
  $|z| \geq R$,
    \[   \sum_{R \leq |w| \leq 2R}  G_A(z,w) \leq c \,\sin(\theta_z/2) \,\sqrt{ \frac{R}{|z|}}
  \]
  and hence, using \eqref{jan10.1},
  \begin{eqnarray*}
   \sum_{2^{-n}|z| \leq |w| \leq 2^{-n+1}|z|}G_A(z,w) \, |Lf(w)|
     &  \leq & c  \, \sin(\theta_z/2)  \, 2^{-n/2}
       \,(2^{-n}|z|) ^{-7/2} \\
       & \leq &  c \, \sin(\theta_z/2)  \, 2^{3n}
         |z|^{-7/2} 
         \end{eqnarray*}
If we sum this over all $n \geq 0$ such that $2^{-n} |z|
  \geq 1$, we get
 \[   \sum_{|w| \leq |z|}    G_A(z,w) \, |Lf(w)|     
   \leq c \, \sin(\theta_z/2) \, |z|^{-1/2} = c \, f(z) \, |z|^{-1}.\] 
This, combined with \eqref{jun13.3}, gives   
\eqref{jan10.2}.

A corollary of \eqref{jan10.2} is as follows.  Let $B \subset A$ be finite
and $\tau = \tau_B = \min\{j \geq 0: S_j \not \in B\}$. 
Then, for $z \in B$,
\[     \left| f(z) - \E^z[f(S_{\tau})] \right|
   \leq c \, |z|^{-1/2} \ f(z).\]
 Indeed, if $M_0=f(S_0)$ and 
 \[      M_{n}  =   f(S_n) - \sum_{j=0}^{n-1} \, L f(S_j), \]
 then $M_{n \wedge \tau}$ is a uniformly integrable martingale
 and hence
 \[    \E^z[M_\tau] = \E^z[M_0] = f(z) . \]
 But
 \[   \left| \E^z[M_\tau] - \E^z[f(S_\tau)] \right|
     \leq \sum_{w \in B} G_B(z,w) \, |Lf(w)| \leq c \,f(z) \, |z|^{-1}.\]
This establishes  \eqref{mainclaim} which completes the proof.

 \end{proof}

\section*{Acknowledgements}
The authors are grateful to the referee for reducing the number of errors in this paper and for helping improve its structure. Lawler was supported by National Science Foundation grant DMS-1513036. Viklund was supported by the Knut and Alice Wallenberg Foundation, the Swedish Research Council, the Gustafsson Foundation, and National Science Foundation grant DMS-1308476. 
     
     \bibliographystyle{plain}
\bibliography{TransLERW}
  
\end{document}